\documentclass[12pt]{amsart}
\usepackage{amsmath,amscd}
\usepackage{geometry,amsfonts,amssymb,amsthm,txfonts,pxfonts,amscd} 
\usepackage{algorithmicx,algpseudocode}
\usepackage{algorithm}
\usepackage{multicol}
\usepackage{refcheck,graphicx,url}
\norefnames
\nocitenames
\def\struckint{\mathop{%
\def\mathpalette##1##2{\mathchoice{##1\displaystyle##2}%
  {##1\textstyle##2}{##1\scriptstyle##2}{##1\scriptscriptstyle##2}}%
\mathpalette
{\vbox\bgroup\baselineskip0pt\lineskiplimit-1000pt\lineskip-1000pt
\halign\bgroup\hfill$}
{##$\hfill\cr{\intop}\cr\diagup\cr\egroup\egroup}%
}\limits}
\newtheorem{theorem}{Theorem}[section]
\newtheorem{lemma}[theorem]{Lemma}
\newtheorem{corollary}[theorem]{Corollary}

\newtheorem{definition}[theorem]{Definition}
\newtheorem{fact}[theorem]{Fact}

\newtheorem{proposition}[theorem]{Proposition}

\theoremstyle{remark}

\newtheorem{method}[theorem]{Method}
\newtheorem{remark}[theorem]{Remark}

\newcommand{\cx}{\mathbb{C}}
\newcommand{\integers}{\mathbb{Z}}

\newcommand{\ratls}{{\bf Q}}

\DeclareMathOperator{\Li}{Li}

\DeclareMathOperator{\Sp}{Sp}
\DeclareMathOperator{\SL}{SL}

\DeclareMathOperator{\Sumset}{Sumset}

\begin{document}
\title{Large Galois groups with applications to Zariski density}

\author{Igor Rivin}
\address{Department of Mathematics, Temple University, Philadelphia}
\curraddr{Mathematics Department, Brown University, Providence, RI}
\email{rivin@temple.edu}
\date{\today}
\begin{abstract}
We introduce the first provably efficient algorithm to check if a finitely generated subgroup of an almost simple semi-simple group over the rationals is Zariski-dense. We reduce this question to one of computing Galois groups, and to this end 
we describe efficient algorithms to check if the Galois group of a polynomial $p$ with integer coefficients is ``generic'' (which, for arbitrary polynomials of degree $n$ means the full symmetric group $S_n,$ while for reciprocal polynomials of degree $2n$ it means the hyperoctahedral group $C_2 \wr S_n.$). We give efficient algorithms to verify that a polynomial has Galois group $S_n,$ and that a reciprocal polynomial has Galois group $C_2 \wr S_n.$ We show how these algorithms give  efficient algorithms to check if a set of matrices $\mathcal{G}$ in $\mathop{SL}(n, \mathbb{Z})$ or $\mathop{Sp}(2n, \mathbb{Z})$ generate a \emph{Zariski dense}  subgroup.
 The complexity of doing this in$\mathop{SL}(n, \mathbb{Z})$   is of order $O(n^4 \log n \log \|\mathcal{G}\|)\log \epsilon$ and in $\mathop{Sp}(2n, \mathbb{Z})$  the complexity is of order $O(n^8 \log n\log \|\mathcal{G}\|)\log \epsilon$ In general semisimple groups we show that Zariski density can be confirmed or denied in time of order $O(n^14 \log \|\mathcal{G}\|\log \epsilon),$ where $\epsilon$ is the probability of a wrong "NO" answer, while $\|\mathcal{G}\|$ is the measure of complexity of the input (the maximum of the Frobenius norms of the generating matrices). The algorithms work essentially without change over algebraic number fields, and in other semi-simple groups. However, we restrict to the case of the special linear and symplectic groups and rational coefficient in the interest of clarity.
\end{abstract}
\thanks{The author would like to thank Nick Katz, Gopal Prasad, Andrei Rapinchuk,Mark van Hoeij,  Yves Cornulier, Emmanuel Breuillard, and Curt McMullen for helpful conversations. I would also like to thank Brown University and ICERM for their hospitality and financial support during the preparation of this paper.}
\maketitle
\section{Introduction}
This paper came from the author's interest in \emph{generic} elements and subgroups in linear groups (the groups $\SL(n, \integers)$ and $\Sp(2n, \integers)$ are particularly interesting in applications. However, this paper may be of interest to people who have no interest in Zariski-dense subgroups. These people should start with Section \ref{planorest} and finish with section \ref{complexity}. There are a number of definitions of "generic" (see the author's article \cite{rivinMSRI}), but luckily they lead to the same qualitative picture. It was shown by the author in \cite{MR2725508} that a subgroup generated by a collection of (at least two) generic elements is \emph{Zariski dense}, and it was shown by Richard Aoun in \cite{aoun2011random} that such a subgroup is \emph{free}, and hence of infinite index (at least in higher rank; a similar result in rank one was shown by E. Fuchs and the author). Put together, this means that a generic subgroup of a lattice is a \emph{thin group} in Sarnak's terminology (see Sarnak's survey \cite{sarnakMSRI}). 

Now, the question arises of how we would actually show that the span $H$ of a finite collection of matrices is thin. The first step is deciding whether $H$ is Zariski dense in some (semisimple) algebraic group $G.$ There are a couple of natural approaches to this:

\begin{method}
\label{eskinmethod}
 See whether the logarithms of the elements in $H$ generate the Lie algebra of $G.$ This method has been used  by A. Eskin, G. Forni, C. Matheus, A. Zorich \cite{forni2012zero}, and seems to work reasonably well in the cases they tried, but seems to come with many theoretical and practical difficulties (the logarithms are not uniquely defined and are real numbers, so precision is an issue. It is also not clear how many (and which) elements of $H$ to generate).
\end{method}
\begin{method}
\label{weismethod}
Namely, it was shown by C.~Matthews, L.~Vaserstein, and B.~Weisfeiler in \cite{MR735226} that for a Zariski-dense subgroup \emph{all but a finite number} of congruence projections are surjective. Conversely, it was shown by T. Weigel in \cite{MR1600994}, that if \emph{a single} projection for modulus greater than $3$ is surjective, than the group is Zariski dense. The algorithm is now clear: compute the projection of our group modulo primes $3, 5, \dotsc.$ If any of them is surjective (and this can be computed in probabilistic polynomial time using, for example, the algorithm of P.~Neumann and C.~Praeger, see \cite{MR1182102}), then we return true. The question is: when do we return \textbf{false}? The original results of \cite{MVW} are not effective (though the authors claim that they can be so made). The first effective bound was written down by Emmanuel Breuillard in \cite[Theorem 2.1]{breuillard2014approximate}. Unfortunately, the bound is exponential in the size of the input, so this algorithm runs in worst case exponential time.
\end{method}
\begin{method}
\label{grobnermethod}
In the paper \cite{DerksenJeandelKoiran}, H.~Derksen, E.~Jeandel, and P.~Koiran describe an algorithm to compute the Zariski closure of a finitely generated group. Their algorithm uses Gr\"obner basis techniques, and (as described) the complexity is worse than exponential - the authors note, however (see Remark 12 in \cite{DerksenJeandelKoiran}) that in characteristic 0 one might be able to improve the performance.
\end{method}

In this note we use a different idea, which leads to an algorithm, which, while still assymetrical (the answer \textsc{yes} indicates 100\% certainty, while the answer \textsc{no} indicates that the probability of an affirmative answer is not bigger than $\epsilon,$ and if we are willing to spend twice as long, the probability of being wrong goes down to $\epsilon^2.$

The idea is based on the following result showed in \cite{rivin2008walks} for $G=\SL(n, \integers),$ and $G = \Sp(2n, \integers)$.\footnote{The result is stated in \cite{rivin2008walks} for $\SL(n, \integers),$ while the only result stated for $\Sp(2n, \integers)$ is that the charactersistic polynomial of a generic matrix is irreducible. However, the argument goes through immediately for the symplectic group (indeed, it is easier than for $\SL(n, \integers)$). The result was stated for $\Sp(2n, \integers)$  in \cite{kowalski2008large}[Theorem 7.12] }
\begin{theorem}
\label{genericgal}
Let $M$ be a generic element of $G$ as above. Then, the Galois group of the characteristic polynomial of $M$ equals the Weyl group of $G.$ Furthermore, if ``generic'' is taken to be with respect to the uniform measure on random words of length $N$ in a symmetric generating set of $G,$ then the probability that $M$ \emph{fails} to have the requisite Galois group of characteristic polynomial goes to zero exponentially fast with $N.$
\end{theorem}
It should be remarked that the Weyl group of $\SL(n, \integers)$ is the symmetric group $S_n,$ while the Weyl group of $\Sp(2n, \integers)$ is the signed permutation group (also known as the hyperoctahedral group) $C_2 \wr S_n.$
For general semisimple groups Theorem \ref{genericgal} is still true, and the key number-theoretic argument can be found in the foundational paper \cite{prasad2003existence} by G. Prasad and A. Rapnichuk. The exponential convergence argument then works by the argument of \cite{rivin2008walks,rivin2009walks}. This result has been rediscovered (using exactly the same method) by Jouve, Kowalski, Zywina \cite{jouve2010splitting}.

Theorem \ref{genericgal} was extended to the situation where $M$ is a generic element of a \emph{Zariski dense subgroup} $\Gamma$ of $G$ (as above) in \cite{MR2725508}; see also \cite{lubotzky2012galois}. So, if a subgroup $\Gamma$ is Zariski dense, then long words in generators of $\Gamma$ have the full Galois group of the Zariski closure with probability exponentially close to $1.$ This leads us to use the Galois group as the tool to check Zariski density -- if we check that the Galois group of a random long element is \emph{not} the Weyl group of the putative Zariski closure, then $\Gamma$ is not Zariski dense, with probability exponentially close to $1$ as a function of $N.$ This is already sufficient for practical purposes: we use one of the Methods \ref{eskinmethod},\ref{weismethod} to try to verify Zariski density, while using the Galois group method to check non-density (assuming that we can quickly check that the Galois group of the characteristic polynomial is ``large'', which is the subject of much of this paper). This method turns out to be not quite correct, and also gives no running time guarantee. Luckily, the situation is completed by the following nice result of G. Prasad and A. Rapinchuk:
\begin{theorem}[\cite{prasadrapinchukMSRI}[Theorem 9.10]]
If $G$ is $\SL(n, \integers)$ and $\Gamma < G$ is such that it contains one element $\gamma_1$ with Galois group of characteristic polynomial equal to $S_n$ and another element $\gamma_2$ such that $\gamma_2$ is of infinite order and does not commute with $\gamma_1,$ then $\Gamma$ is Zariski-dense in $G.$ If $G$ is $\Sp(n, \integers),$ and $\Gamma < G$ contains elements $\gamma_1, \gamma_2$ as above, then either $\Gamma$ is Zariski-dense in $G$ or the Zariski closure of $\Gamma$ (over $\cx$) is the product of $n$ copies of $\SL(2, \cx).$
\end{theorem}
We then have the following simple (at least to write down) Algorithm \ref{zalg} to confirm or deny whether a given subgroup $\Gamma = \langle \gamma_1, \dotsc, \gamma_k\rangle$ of $\SL(n, \integers)$ or of $\Sp(2n, \integers)$ is Zariski dense.
\begin{algorithm}
\label{zalg}
\caption{Algorithm to compute whether $\Gamma = \langle \gamma_1, \dotsc, \gamma_k\rangle < G,$where $G$ is one of $\SL(n, \integers)$ and $\Sp(2n, \integers)$ is Zariski dense}
\begin{algorithmic}[1]
\Require $\epsilon > 0$
\Statex
\Function{ZariskiDense}{G, $\epsilon$}
\Comment $G$ is one of $\SL(n, \integers),\Sp(2n, \integers).$
State $W \gets \mbox{Weyl Group of $G$}.$
\Comment $W$ is $S_n$ for $\SL(n, \integers),$ and $C_2 \wr S_n$ for $\Sp(2n, \integers).$
\State $N\gets c \log \epsilon$
\Comment $c$ is a computable constant.
\State $w_1\gets \mbox{a random product of $N$ generators of $\Gamma.$}$
\State $w_2\gets \mbox{a random product of $N$ generators of $\Gamma.$}$
\If{ $w_1$ commutes with $w_2$}
\State \Return FALSE
\EndIf
\State $\mathcal{G} \gets \mbox{the Galois group of the characteristic polynomial of $w_1$}$
\If  {$\mathcal{G}  \neq W$}
\State \Return FALSE
\EndIf
\State $\mathcal{G} \gets \mbox{the Galois group of the characteristic polynomial of $w_2$}$
\If {$\mathcal{G} \neq W$}
\State \Return FALSE
\EndIf
\If {$G = \SL(n, \integers)$}
\State \Return TRUE
\EndIf
\If {$\Gamma$ acts irreducibly on $\cx^{2n}$}
\State \Return TRUE
\EndIf
\State \Return FALSE
\EndFunction
\end{algorithmic}
\end{algorithm}

The rest of the paper is dedicated to an elaboration and analysis of the steps constituting Algorithm \ref{zalg}, as well as Algorithm (\cite{zalg2}, which works for arbitrary semisimple groups, but is considerably less efficient.

\subsection{Plan of the rest} 
\label{planorest}The least clear part is how to check that the Galois group of the characteristic polynomials in question is the Weyl group of the ambient group. Now, every monic polynomial if degree $n$ with integer coefficients arises as the characteristic polynomial of some matrix in $M^{n\times n}(\integers),$ while characteristic polynomials of symplectic matrices are \emph{reciprocal} (the coefficient sequence reads the same forwards and backwards, or, equivalently, $x^{2n} f(1/x) = f(x).$ It follows that the roots of such a polynomial come in inverse pairs, and so it is not hard to see that the Galois group of such a polynomial is the \emph{signed permutation} or the \emph{hyperoctahedral group} $H_n = C_2 \wr S_n$ (this group acts on the set $\{1, 2, \dotsc, 2n\}$ by permuting the blocks of the form $\{2i, 2i+1\}$ and possibly transposing the elements of a block). 

The rest of the paper is organized as follows:

In Section \ref{perms} we give some necessary background on permutation groups.

In Section \ref{polysec} we give some background on polynomials and their Galois groups.

In Section \ref{chebfrob} we discuss the Frobenius and Chebotarev density theorems.

In Section \ref{galoisgen} we discuss algorithms (old and new) to compute Galois group (big and small), concentrating on the big. We concentrate on very big (the symmetric and alternating groups). The hyperoctahedral group gets its own Section \ref{complexity}

We come back to Algorithm \ref{zalg} n Section \ref{zalgsec}, and show that the complexity for $SL(n, \integers)$ is $O(n^4 \log n)$ and in $\SL(2n, \integers)$ it is of order $O(n^8 \log n).$

In Section \ref{zalg2sec} we describe Algorithm \ref{zalg2} and, show that while it  is  simpler and runs in polynomial time, it is much less efficient than Algorithm \ref{zalg} in the cases where they both work.

\section{Some lemmas on permutations}
\label{perms}
Consider a subgroup $G$ of $S_{2n}$ which is known to be a subgroup of $H_n = C_2 \wr S_n.$ We have the following Proposition:
\begin{proposition}
\label{charprop}
The $G$ is all of $H_n$ if and only if it surjects under $S_n$ under the natural projection and it contains a transposition.
\end{proposition}
\begin{proof}
If $G = H_n$ then clearly it surjects and it does contain a transposition, so we need to show the other direction. The first observation is that if $\tau \in G$ is a transposition, then it must be supported on a block. Indeed, if (without loss of generality) $\tau(1) = 3,$ then $\tau(2) = 4,$ contradicting the fact that $\tau$ only moves two elements. Since $G$ surjects onto $S_n,$ it acts transitively on the blocks, and so it contains \emph{all} the transpositions supported on the blocks, and all of their products, so (by counting) it contains the entire kernel of the natural projection $H_{2n} \rightarrow S_n,$ whence, by counting elements, $G$ must be all of $H_n.$
\end{proof}
In order to make Proposition \ref{charprop} useful for our algorithmic purposes, we first state an obvious corollary:
\begin{corollary}
\label{obvcor}
In the statement of Proposition \ref{charprop} it is enough to require that $G$ surject onto $S_n$ and contain an element $\sigma$ whose cycle structure is $(2, n_1, \dots, n_k)$ where all of the $n_k$ are odd.
\end{corollary}
\begin{proof}
The element $\sigma^{\prod n_k}$ is a transposition.
\end{proof}
We finally ask, how many elements $\sigma$ of the type described in Corollary \ref{obvcor} does $H_n$ contain?
As shown in the proof of Proposition \ref{charprop}, the $2$-cycle must be supported on a block. Since there are $n$ blocks, this can be chosen in $n$ ways. On the other hand, all elements in an odd-length cycle $\rho$ must be contained in \emph{distinct} blocks (if not, then, without loss of generality, $\rho^k(1) = 2$ for some $k < |\rho|,$ but then $\rho^{2k} = 1,$ contradicting the requirement that $\rho$ have odd order). That means that each such cycle must have a a "double": if $\rho = (a_1, \dots, a_k),$ then $\rho^\prime = (b_1, \dots, b_k)$ where $b_i$ is in the same block as $a_i$. To summarize, the element $\sigma$ defines, and is uniquely defined by the following data: a block, an element of $S_{n-1}$ of odd order, and a preimage of that element under the natural map $H_{n-1} \rightarrow S_{n-1}.$ There are $n$ possible blocks, and each element of $S_{n-1}$ has $2^{n-1}$ possible preimages, so the number of ``good'' elements $\sigma$ equals $n 2^{n-1}$ times the number of elements of odd order in $S_{n-1}.$ We now need:
\begin{fact}
\label{elcount}
The number $o_n$ of elements of odd order in $S_n$ is asymptotic to $c n!/\sqrt{n}.$
\end{fact}
\begin{proof}

As discussed above, an element of odd order is a product of disjoint odd cycles, so by the standard Flajolet-Sedgewick theory, the expontial generating function of the sequence $o_n$ is given by 
\[
\exp(\sum_{i=0}^\infty \frac{z^{2i+1}}{2i+1}) = \sqrt{\frac{z+1}{z-1}},
\]
so the statement of  Fact i\ref{elcount} s equivalent to the statement that the coefficients of the power series defining $
\sqrt{\frac{z+1}{z-1}}$ are asymptotic to $c/\sqrt{n}.$ This can be done using the machinery of asymptotics (a la Flajolet-Odlyzko \cite{flajolet1990singularity}, see Flajolet-Sedgewick \cite{flajolet2009analytic} or Pemantle-Wilson \cite{pemantle2013analytic} for details), but in this case it is much simpler:
\[\sqrt{\frac{z+1}{z-1}} = (z+1) (z^2-1)^{-1/2}.\]
Thinking of $(z^2-1)^{-1/2}$ as a function of $z^2,$ the binomial theorem tells us that the coefficients are positive and are asymptotic to $1/\sqrt{n}.$ Multiplying that power series by $z+1$ produces a power series does not change the even coefficients, and makes the adjacent even and odd coefficients equal.
\end{proof}
\begin{remark} The coefficient $c$ is approximately equal to $0.8.$\end{remark}
\begin{corollary}
The number of elements in $S_n$ which have one transposition and the rest of cycles odd is asymptotic to $\frac{cn!}{2\sqrt{n-2}}.$
\end{corollary}
\begin{proof}
A transposition is determined by its support. For each choice of a pair of elements $a, b,$ we know that there are $c(n-2)!/\sqrt{n-2}$ elements with the transposition $(a\ b)$ and the rest of cycles odd. Since there are $n(n-1)/2$ choices of pair $(a, b)$ we get $\frac{cn!}{2\sqrt{n-2}},$ as advertised.
\end{proof}
\begin{corollary}
The number of elements in $H_n$ of the type described in Corollary \ref{obvcor} is asymptotic to 
\[
\frac{c n! 2^n}{2 \sqrt{n-1}}.
\]
\end{corollary}
\begin{proof}
Follows immediately from Fact \ref{elcount} and the discussion immediately above it.
\end{proof}

The next fact  is universally useful.
\begin{lemma}
\label{translemma}
Suppose $G < S_n$ be a $k$-transitive group, is such that the stabilizer (in $G$) of every $k$-tuple of points is a transitive subgroup of $S_{n-k}.$ Then $G$ is $k+1$ transitive.
\end{lemma}
\begin{proof}
 Given points $a_1, a_2, \dots, a_{k+1}, b_1, b_2, \dots, b_{k+1}$ we would like to send $a_i$ to $b_i$ for $i=1, \dots, k+.$ Since $G$ is $k$-transitive, there is a $g\in G$ such that $g(a_i) = b_i,$for $i\leq k.$ Since the stabilizer of $b_1, \dots, b_k$ is transitive by hypothesis, there is $h\in G$ such that $h(b_i) = b_i, $ for $i=1, \dots, k, $ and $h(g(a_{k+1})) = b_{k+1}.$
\end{proof}
We will also need another fact. First a definition: 
\begin{definition}
A permutation group $G$ acting on a set $X$ is called \emph{$k$-transitive} if the induced action on $X^k$ is transitive. A $1$-transitive group is simply called transitive.
\end{definition}
Now the fact:
\begin{fact}
\label{camfact}
Every finite $6$-transitive group is either $A_n$ or $S_n.$ The only other $4$-transitive groups are the Mathieu groups $M_{11}, M_{12}, M_{23}, M_{24}.$
\end{fact}
The statement with some indication of the proof strategy can be found in \cite[page 110]{cameron1999permutation}, though a much more detailed explanation can be found in Cameron's 1981 survey paper \cite{cameron1981finite}.
\subsection{C. Jordan's theorem}
A very useful result in detection of large Galois groups is C. Jordan's theorem:
\begin{theorem}[\cite{isaacs2008finite}[Theorem 8.18]
\label{jordanthm}
Let $\Omega$ be a finite set, and let $G$ act primitively on $\Omega.$ Further, let $\Lambda \subseteq \Omega$ with $|\Lambda| \leq |\Omega|-2.$ Suppose that $G_\Lambda$ (the subgroup of $G$ stabilizing of $\Lambda$) acts primitively on $\Omega \backslash \Lambda.$ Then, the action of $G$ on $\Omega$ is $|\Lambda|+1$-transitive.
\end{theorem}
We will be using some corollaries of Theorem \ref{jordanthm}
\begin{corollary}[\cite{isaacs2008finite}[Theorem 8.17]]
\label{transthm}
Let $G$ be a permutation group acting primitively on $\Omega$ and containing a transposition. Then, $G$ is all of $S_\Omega.$
\end{corollary}
\begin{corollary}
\label{longcycle}
Let $G$ be a permutation group acting primitively on $\Omega$ and containing a cycle of \emph{prime}  length  $l$ , with $|\Omega|/2 < l < |\Omega| - 4.$ Then $G$ is either $A_\Omega$ or $S_\Omega.$
\end{corollary}
\begin{proof}
If we knew that $G$ acted primitively on $\Omega,$ it would follow that the action is $6$-transitive, whereupon the result would follow from Fact \ref{camfact}. However, the existence of a long prime cycle tells us that the action of $G$ is primitive.
\end{proof}
\begin{corollary}[Corollary of Corollary \ref{longcycle}]
It is enough to assume that $G$ has an element$g$ whose cycle decomposition \emph{contains} a cycle of length $l$ satisfying the hypotheses of Corollary \ref{longcycle}.
\end{corollary}
\begin{proof}
Obviously, all the short cycles in the cycle decomposition of $g$ are relatively prime to $l.$ So, raising $g$ to the least common multiple of the lengths of the short cycle will produce an $l$-cycle.
\end{proof}
\begin{remark}
It is an easy consequence of the prime number theorem that the density of elements satisfying the hypothesis of Corollary \ref{longcycle} is asymptotic to $\log2/\log n,$ for $n$ large. 
\end{remark}

\section{A bit about polynomials}
\label{polysec}
In this section we will discuss some useful facts about polynomials. First, reciprocal polynomials: Recall that a reciprocal polynomial is one of the form $f(z) = \sum_{i=0^n} a_i x^i,$ where $a_{n-i} = a_i$ for all $i.$ In the sequel we will always assume that the polynomials are monic, have integer coefficients (unless otherwise specified) and (when reciprocal) have even degree. A reciprocal polynomial $f(x)$ of degree $n$ satisfies the equation $x^n f(1/x) = f(x),$ and therefore the roots of $f(x)$ (in the splitting field of $f$) come in pairs $r, 1/r.$ To a reciprocal polymial $f$ we can associate the \emph{trace polynomial} $F$ of half the degree, by writing $f$ in terms of the variable $z = x + 1/x$ (constructing the trace polynomial is a simple matter of linear algebra, which we leave to the reader). While the Galois group  $G(f)$ of a reciprocal polynomial $f$  (of degree $2n$ now) is a subgroup of the hyperoctahedral group $C_2 \wr S_n,$ the Galois group of the associated trace polynomial $F$ is the image of $G(f)$ under the natural projection to $S_n$ -- see \cite{vianaveloso} for a very accessible introduction to all of the above.

Another piece of polynomial information we will need is a bound on the discriminant of the polynomial. The best such bound is due to K.~Mahler \cite{mahlerdisc}: His result is the following:
\begin{theorem}[Mahler, \cite{mahlerdisc}]
\label{mahler}
Let $f(x) = \sum_{i=0}^n a_i x^i \in \cx[x],$ and let $|f|_1 = \sum_{i=0}^n |a_i|.$ Then the discriminant $D(f)$ of $f$ has the following bound:
\[
|D(f)| \leq n^n |f|_1^{n-2}.
\]
\end{theorem}

\section{The Frobenius Density Theorem}
\label{chebfrob}
Let $f(x)\in \integers[x]$ be a polynomial of degree $n.$ Its Galois group acts by permutations on roots of $f(x)$ and can thus be viewed as a a subgroup of the symmetric group $S_n.$  Reducing $f$ modulo a prime $p$ produces a polynomial $f_p$ with coefficients in the finite field $\mathbb{F}_p,$ which will factor (over $\mathbb{F}_p$) into irreducible factors of degrees $d_1, \dots, d_k,$ with $d_1 + \dots + d_k=n.$ The Galois group of $f_p$  is a cyclic permutation group with cycle structure $(d_1, d_2, \dots, d_k)$ and it is a fundamental fact of Galois theory that the Galois group of $f$ over $\mathbb{Q}$ contains an element with the same cycle type. In fact, a stronger statement is true:
\begin{theorem}[The Frobenius density theorem]
The density (analytic or natural) of the primes $p$ for which the splitting type of $f_p$ is the given type $d_1, \dots, d_k$ is equal to the proportion of elements of the Galois group of $f$ with that cycle type.
\end{theorem}
Since the cycle structure of a permutation gives its conjugacy class in $S_n,$ the Frobenius density theorem talks about conjugacy \emph{in the symmetric group $S_n.$} The stronger Chebotarev density theorem addresses the finer problem of conjugacy \emph{in the Galois group $G$ of $f.$} Even defining the terms needed to state the Chebotarev theorem will take us much too far afield -- the reader is, instead, referred to the beautiful survey paper \cite{stevenhagen1996chebotarev}. The fact of the matter is that for the purpose of computing Galois groups it is the Frobenius density theorem which is used (though this is hard to tell from the literature, which invariably refers to the Chebotarev theorem).

One of the basic methods of computing Galois groups (or at least showing that they are large permutation groups, such as $A_n$ or $S_n,$ as is usually the case) consists of factoring the polynomial modulo a number of primes and seeing which cycles one gets. If one gets enough ``interesting'' cycle types, one knows that the Galois group is $A_n$ or $S_n.$ If, on the other hand, one keeps not getting the cycle types one expects for too many primes, the probability that the Galois group actually contains them becomes exponentially smaller with each prime one checks, so, after a while, one is reasonably sure that the Galois group is ``small''.  Since the Frobenius theorem is an asymptotic statement, it is useless without error bounds. Luckily, such were provided in the foundational paper of J. Lagarias and A. Odlyzko \cite{lagarias1977effective} (their paper concerns the Chebotarev theorem, but we will state it for the Frobenius theorem, since this is all we will every use):
\begin{theorem}[Lagarias-Odlyzko]
\label{lagodthm}
Let $f \in \mathbb{Q}[x]$ be of degree $n$ while its
splitting field $L$ is of degree $n_L.$ Suppose the Generalized Riemann Hypothesis holds for the Dedekind zeta function of $L.$ Let $D(f)$ be the absolute value of the discriminant of $f.$ Then,  if $C$ is a cycle type in $S_n,$ then the number of primes $p$ smaller than $x$ for which the splitting type of $f_p$ corresponds to $C$ (which shall be denoted by $\pi_C(x, L)$) satisfies 
\begin{equation}
\label{lagodest}
\left|\pi_C(x, L) - \frac{|C|}{|G|} \Li(x) \right| \leq c_{LO} \left\{ \frac{|C|}{|G|} x^{\frac12} \log\left(D(f) x^{n_L}\right) + \log D(f)\right\},
\end{equation}
with an effectively computable constant $c.$ 
\end{theorem}
\begin{remark} In his note \cite{oesterle1979versions}, J. Oesterl\'e announced  the following strengthening of the estimate of Eq.~\eqref{lagodest}:
\begin{equation}
\label{oesest}
\left|\pi_C(x, L) - \frac{|C|}{|G|} \Li(x) \right| \leq \frac{|C|}{|G|} \sqrt{x}\left[\log D(f) \left(\frac{1}{\pi} + \frac{5.3}{\log[x]}\right) + n_L\left(\frac{\log x}{2\pi} + 2\right)\right], 
\end{equation}
which would make the constant $c$ of Eq.~\eqref{lagodest} be equal to $\frac{1}{2\pi}.$ Unfortunately, Oesterl\'e never  published the proofs of the announced result. In the sequel, we will use the constant $c_{LO},$ which the reader can think of as $1/(2\pi)$ if she prefers.
\end{remark}
\begin{remark}
In his paper \cite{serrecheb}, J-P Serre notes that one can remove the ``parasitic'' $\log D(f)$ summand from the right hand side of Eq.~\eqref{lagodest}, to get an estimate of the form:
\begin{equation}
\label{serreest}
\left|\pi_C(x, L) - \frac{|C|}{|G|} \Li(x) \right| \leq c_{LO} \left\{ \frac{|C|}{|G|} x^{\frac12} \log\left(D(f) x^{n_L}\right) \right\}.
\end{equation}
\end{remark}
Of course, in order for the bounds in terms of the discriminant to be useful to be useful, we need to know how big the discriminant is, but luckily we do, thanks to Mahler's bound (Theorem \ref{mahler}). We also know that the degree of the splitting field of $f(x)$ is at most $n!$ (where $n$ is the degree of $f$). Substituting this into Eq. \eqref{serreest}, awe get:
\begin{corollary}
W have the following estimate:
\begin{equation}
\label{lagodexp}
\left|\pi_C(x, L) - \frac{|C|}{|G|} \Li(x) \right| \leq c_{LO}\sqrt{x}\left(|C| \log x +\frac{|C|}{|G| }(n \log n + n \log |f|_1) \right).
\end{equation}
\end{corollary}
\section{Some Galois group algorithms}
\label{galoisgen}
Galois groups of polynomials are believed to be difficult to compute in general, though there is no current agreement as to where the computation of Galois group falls in the complexity hierarchy (which is, of course, itself conjectural). 
\subsection{Kronecker's algorithm} The most obvious algorithm (which, in essence, goes back to Galois, but was first published by Kronecker in his book \cite{kronecker1882grundzuge}) is the following: 
\begin{algorithm}
\label{kroneckeralg}
\caption{naive algorithm for computing Galois groups}
\begin{algorithmic}[1]
\Require $f\in \integers[x]$ a polynomial of degree $n,$ with roots $r_1, \dots, r_n.$
\Statex
\Function{Kronecker}{$f$}
\State $R(x) \gets \prod_{\sigma \in S_n} x-\sum X r_{\sigma(i)}.$
\Comment $R(x)$ is the \emph{resolvent} of $f.$
\State $(R_1, \dots, R_k)$ are factors of $f$ over $Q[X_1, \dots, X_n].$
\Return $\mbox{Stabilizer}(R_1) \subset S_n.$
\EndFunction
\end{algorithmic}
\end{algorithm}
Note that the coefficients of the resolvent are symmetric functions of the roots of $f,$ and are thus rational. Nevertheless, since the algorithm involves factoring a polynomial in $n$ variables, and of degree $n!$ this method is not likely to be practical for any but the smallest values of $n.$ 
\subsection{Stauduhar's Algorithm} The next algorithm (as far as I can tell) was designed by R. Stauduhar in \cite{stauduhar}.  The idea of Stauduhar's algorithm (which we describe for \emph{irreducible} polynomials,for simplicity) is that instead of computing the resolvent of $f$ with respect to the full candidate Galois group ($S_n$ in Kronecker's algorithm) we compute it with respect to a set of coset representatives of the largest possible Galois group with respect to a subgroup, then iterate. To be precise:
\begin{algorithm}
\label{stauduhar}
\caption{Stauduhar's algorithm for computing Galois group.}
\begin{algorithmic}[1]
\Require $f \in \integers[x]$ an irreducible polynomial of degree $n,$ with roots $r_1, \dots, r_n.$
\Statex
\Function{findMax}($G, f, \mathcal{F}, \mathcal{M}$)
\For{$M \in \mathcal{M}$}
\State $F_M \gets \sum_{\sigma \in M} \sigma(\mathcal{F}).$
\Comment $F_M$ is invariant under all permutations in $M,$ and only those.
\State $Q_{G:M}(y; x_1, \dots, x_n) = \prod_{\mbox{coset representatives of $M$ in $G$}} (y-\sigma(F_M(r_1, \dots, r_n)).$

\Comment To compute $Q_{G:M}$ we use the \emph{approximate roots $\tilde{r},$} and then round the coefficients to nearest integers.
\If{$Q_{G:M}$ has an integer root} 
\Return $M$
\EndIf
\EndFor

\Return \textsc{none}
\EndFunction
\Function{Stauduhar}{$f$}
\State $\tilde{r}_f \gets (\tilde{r_1}, \tilde{r_2}, \dots, \tilde{r_n},$ where $\tilde{r_i}$ are high precision numerical approximations to the roots of $f.$
\State $\mathcal{F} \gets x_1 x_2^2 \dots x_n^n.$
\State $G \gets S_n.$
\While{$|G| > 1$}
\State $\mathcal{M} \gets \mbox{the list of conjugacy classes of maximal transitive subgroups in $G.$}$
\State $H \gets \mbox{findMax}(G, f, \mathcal{F}, \mathcal{M}).$
\If{$H$ is \textsc{none}}
\Return $G$
\EndIf
\State $G=H.$
\EndWhile
\EndFunction
\end{algorithmic}
\end{algorithm}
A casual examination will show that Stauduhar's algorithm is designed very much as a technical too, suitable for the state of computers in the late 1960s-early 1970, and has really no fundamental computational complexity advantage over Kronecker's algorithm. Note for example that if the subgroup $M$ is large, then the invariant polynomial $F_M$ will have a lot of terms, while if $M$ is small in comparison with the ambient group $G,$ the degree of the resolvent $Q$ will be huge. In many case we suffer from both problems at once. For example, the wreath products $S_k \wr S_n$ are maximal subgroups of $S_{kn}$ -- in this case both the order and the index are enormous. Curiously, the most glaring problem with Stauduhar's algorithm (the need to know ALL the maximal transitive subgroup structure of permutation groups) is, in a way, the least significant, since the number of such is quite small (low degree polynomial in $n$). That said, Stauduhar's algorithm certainly provides a considerable speedup for the computation of Galois group of low-degree polynomials.
\subsection{Polynomial time (sometimes)}
The first theoretical advance in computing Galois groups came as a consequence of the celebrated Lenstra-Lenstra-Lovasz  (\cite{lenstra1982factoring} )result (at the time viewed as of purely theoretical interest) which showed that one could factor polynomials over the rational numbers in time \emph{polynomial} in the input size -- the algorithm was based on the Lovasz lattice reduction algorithm (nowdays invariably referred to as the LLL algorithm, but attributed to Lovasz in the source). The reason that the algorithm was not viewed as practical at the time was that running time was of the order of $O(n^{9+\epsilon} + n^{7+\epsilon} \log^{2+\epsilon}(\sum a_i^2)),$ for any $\epsilon,$ but with constants depending on the $\epsilon.$ In any case, it was then showed by Susan Landau in \cite{landau1985factoring} that one could also factor in polynomial time over extensions of $\mathbb{Q}$-- Landau's algorithm is basically a combination of the LLL algorithm with a very simple idea going back to Kronecker by way of Barry Trager. Namely, to factor a polynomial over $\mathbb{Q}[\alpha]$ it is enough to factor its norm (which is a polynomial over $\mathbb{Q}$), as long as the norm is square-free. However, the observation is that there is at most a quadratic (in the degree of the extension) shifts of the variable which give non-square-free norms. One does, however need to compute the norm, and then compute a bunch of polynomial gcds, so after the smoke clears, the complexity of factoring a polynomial $f(x)$ of degree $n$ over $\mathbb{Q}(\alpha)$ where $\alpha$ has minimal polynomial $g(y)$ of degree $m$ is:
\[O(m^{9+\epsilon} n^{7+\epsilon} \log^2 |g| \log^{2+\epsilon}(|f| (m |g|)^n (m n)^n)),\] so charitably dropping the $\epsilon$s and terms logarithmic in degrees we get the complexity to be worse than $O((mn)^9 \log^2 |f| \log^2 |g|).$ For example, in the generic case if we adjoin $k$ roots of $f$ and attempt to factor $f$ over the resulting extension, complexity will be worse than $O(n^{9(k+1)} \log^4|f|).$

In any event, Landau notes the following corollary:
\begin{corollary}
\label{landcor}[S. Landau, \cite{landau1985factoring}] The Galois group of a polynomial $f$ can be computed in time polynomial in the degree of  the splitting field $f$ and $\log |f|.$
\end{corollary}
The corollary essentially follows from the polynomial bound on factoring, since, as she shows, one can construct the splitting field by adjoining one root at a time. Notice that Corollary \ref{landcor} is useless for recognizing large Galois groups (since, for example, in the generic case when the Galois group is the full symmetric group, the degree of the splitting field is $n!$), but it \emph{can} be used for polynomial time algorithms to determine whether the Galois group is solvable (this was done by S. Landau and G. Miller in \cite{landaumiller}), thanks in no small part to the Palfy's theorem (\cite{palfy}), which states that the order of a transitive primitive solvable subgroup of $S_n$ is at most $24^{\frac13} n^{3.24399\dots}.$ Solvability has considerable symbolic value, since solvability of the Galois group is equivalent to $f$ being solvable by radicals, and the question of whether all equations were thus solvable was what led to the discovery of Galois theory. More recently, some of these ideas were applied to the design of a polynomial-time algorithm to check if the Galois group of a polynomial is nilpotent (the Landau-Miller algorithm does not go through for this case, though Palfy's bound obviously still holds) by V. Arvind and P. Kurur in \cite{arvindkurur}.

The first polynomial-time algorithm to check that the Galois group of a polynomial is the symmetric group $S_n$ or the alternating group $A_n$ is due to Susan Landau, and it uses Fact \ref{camfact} and Lemma \ref{translemma} The algorithm proceeds as follows:
\begin{algorithm}
\label{landaualg}
\caption{S. Landau's algorithm for large Galois groups}
\begin{algorithmic}[1]
\Require $f\in \integers[x]$ a polynomial of degree $n.$ 
\Statex
\Function{Landau}{$f$}
\State $K\gets \mathbb{Q}.$
\State $R \gets \mbox{roots of $f$}$
\ForAll{$S \subset  R; |S| \leq 5$, in increasing order of size}
\If {$f$ is not irreducible over $K[S]$}
\Return \textsc{false}
\EndIf
\EndFor
\State $D\gets \mbox{discriminant of $f$}$
\If {$D$ is a square of an integer.}
\Return $A_n.$
\EndIf

\Return $S_n.$
\EndFunction
\end{algorithmic}
\end{algorithm}

The algorithm as written looks worse than it is, since once we know that the Galois group is transitive, we only need to check transitivity for of $f$ over $\mathbb{Q}[\alpha]$ for one root of $f,$ and so on. Still, the complexity lies in the last step, where we need to factor $f$ (which is still of degree $n-5$) over an extension of degree $O(n^5),$ which gives us complexity of order $O(n^{45}),$ which is not practical. Since sporadic groups don't occur for large $n,$ we can do better (by checking just $4$-transitivity, which gives us an algorithm of complexity $O(n^{27}),$ which is still not realistic in practice -- note that if the  implied constant is $1,$ then computing the Galois group of a quartic will take some $2^54\sim 2 10^{16}$ operations, or several years on a modern computer.

\subsection{A much simpler algorithm}
A much simpler deterministic algorithm to check if the Galois group of a monic polynomial of degree $d$  in $\mathbb{Z}[x]$ is large (either the symmetric group $S_d$ or the alternating group $A_d$) was discovered by the author. The main ingredient is the \emph{Livingstone-Wagner  theorem}. First a definition:
\begin{definition}
\label{homo}
We say that a permutation group $G_n \leq S_n$ is $k$-homogeneous if $G$ acts transitively on the set of \emph{unordered} $k$-tuples of elements of $\{1, \dots, n\}.$
\end{definition}
\begin{theorem}[Livingstone-Wagner \cite{MR0186725}]
\label{lv}
If (with notation as in definition \ref{homo}) the group $G_n$ is $k$-homogeneous, with $k\geq 5$ and $2 \leq k \leq \frac12 n,$ then $k$ is $k$-transitive.
\end{theorem}
\begin{remark} Obviously the hypotheses of Theorem \ref{lv} can only be met for $n\geq 10.$ \end{remark}
Now, let $M$ be a linear transformation of a complex $n$-dimensional vector space $V^n.$ We define $\bigwedge^k M$ to be the induced transformation on $\bigwedge^k V^n.$ The following is standard (and easy):
\begin{fact}
The eigenvalues of $\bigwedge^k M$ are sums of $k$-tuples of eigenvalues of $M.$
\end{fact}
To avoid excessive typing we will call the characteristic polynomial of $\bigwedge^k M$ by $\chi_k(M).$
\begin{lemma}
\label{kirred}
If the Galois group of $\chi(M)$ is $A_n$ or $S_n,$ and $k<n,$ then the Galois group of $\chi_k(M)$ is $A_n$ or $S_n.$ In particular, $\chi_k(M)$ is irreducible.
\end{lemma}
\begin{proof} The Galois group of $\chi_k(M)$ is a normal subgroup of the Galois group of $\chi(M),$ and so is either $A_n, S_n, or \{1\}.$ In the first two cases we are done. In the last case, the fact that the sums of $k$-tuples of roots of $\chi(M)$ is rational tells us that the Galois group of $\chi(M)$ is a subgroup of $S_k,$ contradicting our assumption.
\end{proof}

\begin{lemma}
\label{khom}
 Suppose that $\chi_k(M)$ is irreducible. Then the Galois group of $\chi(M)$ is $k$-homogeneous.
\end{lemma}
\begin{proof}
The Galois group of $\chi_k(M)$ is a subgroup of the Galois group of $\chi(M).$ Since the roots of $\chi_k(M)$ are distinct, it follows that the Galois group of $\chi(M)$ acts transitively on unordered $k$-tuples of roots of $\chi(M),$ so is $k$-homogeneous. Therefore, so is the Galois group of $\chi(M).$
\end{proof}
Finally, we can state our result:
\begin{theorem}
\label{wedgethm}
Suppose $n>24.$ Then the Galois group of $\chi(M)$ is $S_n$ if and only if $\chi_5(M)$ is irreducible, and the discriminant of $\chi(M)$ is not a perfect square. If $\chi_5(M)$ is irreducible and the discriminant of $\chi(M)$ \emph{is} a perfect square, then the Galois group of $\chi(M)$ is $A_n.$ For $24 \geq n \geq 12,$ the same statements hold with $5$ replaced by $6$ throughout.
\end{theorem}
\begin{proof} Immediate from Lemmas \ref{kirred} and \ref{khom} combined with the Livingstone-Wagner Theorem \ref{lv}.
\end{proof}
Theorem \ref{wedgethm} immediately leads to the following algorithm for checking if the Galois group of a polynomial $p(x) \in \integers[x]$ is large (one of $A_n$ or $S_n.$).
\begin{algorithm}
\label{wedgealg}
\caption{New algorithm for large Galois groups}
\begin{algorithmic}[1]
\Require $n > 10.$
\Require $f\in \integers[x]$ a polynomial of degree $n.$ 
\Statex
\Function{Wedge}{$f$}
\If $n>24$ 
\State $k \gets 5.$
\Else $k \gets 6.$
\EndIf
\State $M\gets\mbox{companion matrix of f}.$
\State $g \gets \chi_k(M).$
\If {$g$ is not irreducible}
\Return \textsc{false}
\EndIf
\State $D\gets \mbox{discriminant of $f$}$
\If {$D$ is a square of an integer.}
\Return $A_n.$
\EndIf

\Return $S_n.$
\EndFunction
\end{algorithmic}
\end{algorithm}

The complexity of Algorithm \ref{wedgealg} can be bounded above using M. van Hoeij's algorithm for factoring univariate polynomials over $\mathbb{Q}.$ The complexity of factoring a polynomial of degree $n$ and  height $h$ (where $h$ is the log of the maximal coefficient)  (see van Hoeij and Novocin \cite{MR2891235}) of factoring over $\mathbb{Q}$  is given by $O(n^6 + h^2 n^4),$ which gives us:
\begin{theorem}
The complexity of Algorithm \ref{wedgealg} is at most $O(n^{30} + h^2 n^{20}).$
\end{theorem}
\begin{remark} van Hoeij algorithm is far more efficient in practice than the running time bounds would indicate, and therefore so is Algorithm \ref{wedgealg}. Nonetheless, it is extremely unlikely that it would be competitive with the probabilistic algorithms described in the next section. It does, however, have the advantage of being fully deterministic (and very simple, to describe and to implement -- the hardest part of the implementation is computing $\chi_k(M)$ without writing out the $k$-th exterior power of the companion matrix.)
\end{remark}
\section{Probabilistic algorithms}
\label{probalg}
Given the rather unsatisfactory state of affairs described above, it is natural to look for other, more practical methods, and, as often in life, one needs to give up something to get something. The algorithms we will describe below are probabilistic in nature -- they involve random choices, and so they are \emph{Monte Carlo} algorithms. They are also one-sided. That is, if we ask our computer "Is the Galois group of $f$ the full symmetric group?", and the computer responds "Yes," we can be sure that the answer is correct. If the computer responds "No", we know that the answer is correct with some probability $0<p<1.$ If we are unsatisfied with this probability of getting the correct answer, we can ask the computer the same question $k$ times. At the end of this process, the probability of getting the wrong answer is (at most) $(1-p)^k,$ so this is what is known as a \emph{Monte Carlo} algorithm.

Not surprisingly, the probabilistic algorithms for Galois group computation are based on the Chebotarev (really Frobenius) density theorem, in the (supposedly) effective form of Eq. \eqref{lagodest}. The idea is that once we know something about the statistics of the permutations in our putative Galois group, we can conduct some probabilistic experiments and predict the statistical properties of the outcome.

Since the Frobenius Density Theorem talks about the conjugacy classes in the symmetric group $S_n$ (recall that the conjugacy class of an element is given by its cycle structure), while the Chebotarev Density Theorem talks about the conjugacy classes in the Galois group itself, it is natural to ask what we can learn about the group just by looking at the conjugacy classes. To this end, we make the following definition (which I believe was first made by J. D. Dixon):
\begin{definition}
A collection of conjugacy classes $C_1, \dots, C_k$ of a group $G$ \emph{invariably generates} $G$ if any collection of elements $c_1, \dots, c_k$ with $c_i \in C_i$ generates $G.$
\end{definition}
We make another definition (of our own, though the concept was also used by Dixon).
\begin{definition}
A collection of conjugacy classes $C_1, \dots, C_k$ of a group $G$ acting on a set $X$ is \emph{invariably transitive} if any collection of elements $c_1, \dots, c_k$ with $c_i \in C_i$ generates a subgroup $H$ of $G$ acting transitively on $X.$
\end{definition}
Dixon (\cite{dixoninv}) shows that in the case that $G$ is a symmetric group $S_n$ acting in the usual ways on $\{1, \dots, n\},$ for a fixed $\epsilon,$ $O_\epsilon(\log^{1/2} n)$ elements will be invariably transitive (we are abusing our own notation: a collection of elements invariably generate if their conjugacy classes do)  with probability $1-\epsilon,$ and a similar estimate for the number of elements which generate $S_n.$ Dixon's proof is quite complicated. The result was improved a couple of years later by Luczak and Pyber in \cite{lupy} -- they removed the dependence on $n,$ but their constant is quite horrifying.

The truth is much more pleasing:
\begin{fact}[\cite{ripe}]
\emph{Four} uniformly random elements in $S_n$ are invariably transitive with probability at least $0.95.$ Therefore, a collection of $4k$ elements is invariably transitive with probability \emph{at least} $1-\frac{1}{20^k}.$
\end{fact}
\begin{proof}[Heuristic argument] First define the \emph{sumset} of a partition of $n$ to be the collection of all sums of subsets of the partition, save $0$ and $n.$  Every permutation $\sigma \in S_n$ defines a partition of $n$ (corresponding to the cycle decomposition of $\sigma$), and defining $s(\sigma)$ to be the sumset of that partition, it is clear that $\sigma_1, \dots, \sigma_k$ are invariably transitive if $\bigcap s(\sigma_i) = \emptyset.$ Now, it is well-known that the expected number of cycles of a random permutation is $\log n$ (the variance is also equal to $\log n$) -- for a nice survey of results of this sort, see the paper by Diaconis, Fulman, and Guralnick \cite{difugu}. Therefore, their collection of sumsets has cardinality $2^{\log n} = n^{\log 2}.$ One can think of $\log 2$ as the "discrete dimension" of the sumset, and so, taking $k$ elements, the dimension of the intersection of the pairs of sumsets with the (big) diagonal is $k(\log 2) - k + 1.$ This is easily seen to turn negative when $k=4.$
\end{proof}

The above point of view is also useful for a very fast algorithm to compute if a collection of elements are invariably transitive: computing the sumset of a partition of $n$ can be done in time $O(n^2)$ by the dynamic programming Algorithm \ref{dyprog}.

\begin{algorithm}
\label{dyprog}
\caption{Dynamic algorithm to compute the sumset of a partition}
\begin{algorithmic}[1]
\Require $X=(x_1, \dots, x_k)$ a collection of positive integers
\Statex
\Function{Sumset}{$f$}
\State $X\gets 0$
\For{$i$ from $1$ to $k$}
$X = X \cup (X+x_i).$
\EndFor

\Return $X.$
\EndFunction
\end{algorithmic}
\end{algorithm}
\section{Probabilistic algorithm to check if $p(x)$ of degree $n$ has Galois group $S_n.$}
\label{symmetriccheck}
To check that the Galois group is the full symmetric group, we first check transitivity, and then use one of two methods to determine whether the group is the full $S_n.$ Both methods are based on C. Jordan's theorem, but one (which has worse complexity in terms of the degree) is used for $n<13,$ and the other is used for $n\geq 13.$ It should be underscored that the algorithm is testing the hypothesis that the group is $S_n$ -- if one has an unknown Galois group, it may be much harder to test whether it acts transitively, for example.
In any case, we will need to first write some helper functions:
\begin{algorithm}
\label{isTrans}
\caption{either confirms that  a polynomial is irreducible over $\ratls,$ or states that the Galois group is \emph{not} the symmetric group}
\begin{algorithmic}[1]
\Require{ $p \in \integers[x]$ of degree $n.$}
\Statex
\Require {$\epsilon>0.$}
\Statex
\Comment $epsilon$ is the probability that we are wrong.
\Function{IsTransititve}{$p, \epsilon$}
\State $s \gets \{1, \dotsc, n\}.$
\For{$i\leq -c \log \epsilon$}
\State
{$q \gets \mbox{random prime with $q \nmid \mathcal{D}(p).$}$}
\State
{$d \gets \mbox{set of degrees of irreducible factors of factorization of $p(x) \mod q$ }$}
\State
$s^\prime \gets \Sumset(d).$
\If{$s^\prime \cap s = \emptyset$}
\State 
\Return IRREDUCIBLE
\EndIf
\State {$s\gets s^\prime$}
\EndFor
\State \Return NOT $S_n.$
\EndFunction
\end{algorithmic}
\end{algorithm}

If the polynomial IS irreducible, we can proceed to the next step. At this point we can assume that the degree of the polynomial is greater than $3.$

\begin{algorithm}
\label{isPrimitive}
\caption{Checking if the Galois group of an irreducible polynomial s either primitive \emph{or} not $S_n$}
\begin{algorithmic}[1]
\Require{ $p \in \mathbb{Z}[x]$ irreducible of degree $n.$}
\Statex
\Require{$\epsilon > 0.$}
\Comment{$\epsilon$ is the probability that we are wrong}
\Function{IsPrimitive}{$p, \epsilon$}
\For{$i\leq - c \log \epsilon \log n$}
\State
{$q \gets \mbox{random prime with $q \nmid \mathcal{D}(p).$}$}
\State
{$d \gets \mbox{set of degrees of irreducible factors of factorization of $p(x) \mod q$ }$}
\If{$d \mbox{ contains a prime bigger than $n>2$}$}
\State
\Return PRIMITIVE
\EndIf
\EndFor
\State \Return NOT $S_n.$
\EndFunction
\end{algorithmic}
\end{algorithm}

Finally, we have either decided that our Galois group is not the full symmetric group or that it is transitive and primitive. We are now down to the last step:
\begin{algorithm}
\label{isSn}
\caption{Checking if the Galois group of an irreducible primitive polynomial is $S_n.$}
\begin{algorithmic}[1]
\Require{ $p \in \mathbb{Z}[x]$ irreducible of degree $n$ with primitive Galois group}
\Statex
\Require{$\epsilon > 0.$}
\Comment{$\epsilon$ is the probability that we are wrong}
\Function{IsSn}{$p, \epsilon$}
\If{$n<13$}
\State
\For{$i\leq - c \log \epsilon$}
\State
{$q \gets \mbox{random prime with $q \nmid \mathcal{D}(p).$}$}
\State
$d \gets \mbox{set of degrees of irreducible factors of factorization of $p(x) \mod q$ }$
\If{$d \mbox{contains one $2$ and the rest of the elements are odd}$}
\State
\Return YES.
\EndIf
\EndFor
\State \Return NO
\Else
\If{ $\mathcal{D}(p)$ is a perfect square}
\State \Return NO
\EndIf
\State
\For{$i\leq c \log \epsilon \log n$}
\State
{$q \gets \mbox{random prime with $q \nmid \mathcal{D}(p).$}$}
\State
$d \gets \mbox{set of degrees of irreducible factors of factorization of $p(x) \mod q$ }$
\If{$d \mbox{contains a prime bigger than $n>2$ and smaller than $n-5$}$}
\State \Return YES
\EndIf
\EndFor
\State \Return NO
\EndIf
\EndFunction
\end{algorithmic}
\end{algorithm}
\subsection{Some remarks on the running time of detecting Galois group $S_n.$}
\label{complexity}
The complexity of (probabilisticalyfactoring a polynomial  of degree $n$ modulo a prime $q$ is bounded by $O(n^{1.816} \log^{0.43} q)$ (see \cite{von2001factoring} for discussion), while the primes we use are (at worst) of order of the size of the discriminant of the polynomial\footnote{in practice we would use much smaller primes}, for a complexity bound of $O(n^{1.1816}(\log^{0.43} n + \log^{0.43}|f|_1)).$   The complexity of computing the discriminant is bounded by $O(n^3(\log n + \log |f|_1).$ This follows from the complexity results of I.~Emiris and V.~Pan \cite{emiris2005improved}. 
Checking a number $r$ for primality can be done in time $\widetilde{O}( \log^2 r)$ (where $\widetilde{O}$ means that we ignore terms of order polynomials in $\log \log r$). Since the probability of a number $r$ is of the order of $1/\log r$ by the prime number theorem, generating a random prime
we factor mod $q$ at most $O(\sqrt{n})$ times, the running time is dominated by computing the discriminant and generating the requisite random primes, and so is of order $O(n^3(\log n + \log |f|_1).$
\subsection{Deciding whether the Galois group of a reciprocal polynomial is the hyperoctahedral group}
\label{hyperalgo}
Deciding whether the Galois group of a reciprocal polynomial is the hyperoctahedral group is (given our preliminaries) an easy extension of the algorithm to check that the Galois group  of a (not necessarily reciprocal polynomial) is the full symmetric group, and is given in Algorithm \ref{hyperoct}. The complexity of this algorithm is, again, dominated by the complexity of computing the discriminant and generating primes, and is again of the order of $O(n^3(\log n + \log|f|_1)).$
\begin{algorithm}
\label{hyperoct}
\caption{Algorithm to check whether the Galois group of a reciprocal polynomial $p(x) \in \integers{Z}[x]$ of degree $2n$ is $C_2 \wr S_n.$}
\begin{algorithmic}[1]
\Require $p(x)$ be a reciprocal polynomial with integral coefficients of degree $2n.$
\Require $\epsilon > 0.$
\Statex
\Function{IsHyperoctahedral}{$p, \epsilon$}
\State $r\gets \mbox{trace polynomial of $p.$}$
\If{The Galois group of of $r$ is \emph{not} the symmetric group $S_n$}
\State \Return NO
\EndIf
\For{$i$ from $1$ to $c \sqrt{n}$}
\State
{$q \gets \mbox{random prime with $q \nmid \mathcal{D}(p).$}$}
\State
$d \gets \mbox{set of degrees of irreducible factors of factorization of $p(x) \mod q$ }$
\If{$d \mbox{contains one $2$ and the rest of the elements are odd}$}
\State \Return YES
\EndIf
\EndFor
\State
\Return NO
\EndFunction
\end{algorithmic}
\end{algorithm}
\section{Back to Zariski density}
\label{zalgsec}
We now return to Algorithm \ref{zalg}. We know how to do every step except for checking that the action on $\mathbb{C}$ is irreducible (which is addressed in Section \ref{irred}). We should now check what the complexity of the algorithm is. On lines 3, 4 we are computing a  (random) product of $N$ matrices. By the Furstenberg-Kesten theorem (or any one of its refinements, see, e.g., \cite{bougerol1985products}) we know that the sizes of the coefficients of the matrices $w_1, w_2$ is of order of $\lambda^N,$ for some $\lambda$ depending on the generating set (but notice that $\lambda$ is at most $\|\log \mathcal{G}\|,$ where $\mathcal{G}$ is the maximum of the Frobenius norms of the generators)., so in our case the coefficients are of the order of $(1/\epsilon)^c,$ for some constant $c,$ and therefore the coefficients of the characteristic polynomial of $w_1, w_2$ are of order of $(1/\epsilon)^{nc}.$ This tells us (using the results in Section \ref{complexity}), that we can check that the Galois group of $w_1, w_2$ is $S_n$ in time $O(n^4 \log \epsilon \log \|\mathcal{G}\|),$  and likewise in the symplectic case we can check that the Galois group is $C_2 \wr S_n$ in the same time. 
\subsection{Testing irreducibility}
\label{irred}
One of the steps in Algorithm \ref{zalg} involves checking that our group acts irreducibly on $V^{2n}.$ This seems hard \emph{a priori}, but there are two ways to deal with this. The first way involves computing in the splitting fields of the characteristic polynomials of the generators, so is not practical. The second, luckily, is polynomial time, and uses Burnside's irreducibility criterion:
\begin{theorem}[Burnside's irreducibility theorem]
\label{burnside}
The only irreducible algebra of linear transformations on a vector space of finite dimension greater than $1$ over an algebraically closed field is the algebra of \emph{all} linear transformations on the vector space.
\end{theorem}
Burnside's theorem is (obviously) classical, and proofs can be found in many places, but the most recent (and simplest) proof by V. Lomonosov and P. Rosenthal is highly recommended, see \cite{lomonosov2004simplest}. Burnside's Theorem \ref{burnside} tells us that in order to check irreducibility, we need only check that the set of all elements in the our group spans the whole matrix algebra $M^{2n \times 2n}$ (thought of as a vector space). Now, our group is infinite, but luckily, Algorithm \ref{irredalg}(suggested by Yves Cornulier) gets around that problem.
\begin{algorithm}<
\label{irredalg}
\caption{Algorithm to check that a finitely generated matrix group acts irreducibly}
\begin{algorithmic}[1]
\Require $X=(g_1, \dots, g_k)$ a collection of generators.
\Statex
\Function{IsIrreducible}{$X$}
\State $V_0\gets \langle I \rangle$
$b_0 \gets \mbox{basis of $V_0$}.$
\Loop
\If {$|b_0| = n^2$}
\State
\Return TRUE
\EndIf
\State
$V_1 \gets \langle b_0 \cup  g_1 b_0 \cup \dots \cup g_k b_0\rangle.$
\State
$b_1 \gets \mbox{basis of $V_1.$}$
\If {$|b_1|= |b_0|$}
\State
\Return FALSE
\EndIf
\State
$V_0 \gets V_1.$
\EndLoop
\EndFunction
\end{algorithmic}
\end{algorithm}
Note that the inner loop of Algorithm \ref{irred} runs \emph{at most} $n^2$ times, and each iteration computes the rank of an at most $n^2\times n^2$ integer matrix, so 
\begin{theorem}
\label{irredcomplex}
The irreducibility of an $n\times n$ matrix group can be decided by using at most $O(n^8\log n\log\|\mathcal{G}\|)$ arithmetic operations (this uses the algorithm of A. Storjohann \cite{storjohann2009integer}).
\end{theorem}

\section{Another Zariski density algorithm}
\label{zalg2sec}
The other method to test Zariski density rests on the following fact:
\begin{fact}
Let $H$ is asubgroup of a semisimple algebraic group $G$ over a field of characteristic $0.$ Then, $H$ is Zariski dense if and only if the following two conditions hold:
\begin{enumerate} 
\item The adjoint representation of $H$ on the Lie algebra of $G$ is irreducible.
\item $H$ is infinite.
\end{enumerate}
\end{fact}
We already know (see Section \ref{irred}) how to determine irreducibility. The additional observation is that an element of finite order has a cyclotomic characteristic polynomial, so the Galois group is cyclic.
This leads to the following Algorithm \ref{zalg2}
\begin{algorithm}
\label{zalg2}
\caption{Algorithm to compute whether $\Gamma = \langle \gamma_1, \dotsc, \gamma_k\rangle < G,$where $G$ is a semisimple algebraic group, is Zariski dense}
\begin{algorithmic}[1]
\Require $\epsilon > 0$
\Statex
\Function{GeneralZariskiDense}{G, $\epsilon$}
\State $N\gets c \log \epsilon$
\Comment $c$ is a computable constant.
\State $w\gets \mbox{a random product of $N$ generators of $\Gamma.$}$
\If{ The characteristic polynomial of $w$ is cyclotomic}
\State \Return FALSE
\EndIf
\If {The adjoint action of $\Gamma$ is irreducible on the Lie algebra of $G$}
\State \Return TRUE
\EndIf
\State \Return FALSE
\EndFunction
\end{algorithmic}
\end{algorithm}
The problem with Algorithm \ref{zalg2} is that the adjoint representation acts on a vector space of dimension $\dim G,$ for the usual classical groups, the running time is going to be of order of $O(n^{14}\log\|\mathcal{G}\|)$ which is much worse than the complexity of Algorithm \ref{zalg} in the cases where they are both applicable.
\bibliographystyle{plain}
\bibliography{msri2}
\end{document}